\documentclass[11pt]{amsart} \usepackage{latexsym, amssymb, stmaryrd}

\newtheorem{thm}{Theorem}[section]
\newtheorem{lemma}[thm]{Lemma}
\newtheorem{prop}[thm]{Proposition}

\newtheorem{cor}[thm]{Corollary}

\theoremstyle{definition}
\newtheorem{df}[thm]{Definition}
\newtheorem{rmk}[thm]{Remark}

\newtheorem{ex}[thm]{Example}
\newtheorem{fact}[thm]{Fact}

\newcommand{\Z}{\mathbb{Z}}

\newcommand{\curly}[1]{\mathcal{#1}}

\newcommand{\E}{\curly{E}}

\newcommand{\G}{\mathcal{G}}
\newcommand{\V}{\mathcal{V}}

\def \r { {\mathbb R} }
\def \<{\langle}
\def \>{\rangle}
\def \n {\mathbb N}
\def \z {{\mathbb Z}}
\def \*Z {{{^*}\Z}}

\def \((  {(\!(}
\def \)) {)\!)}

\def \st {\operatorname {st}}

\def \ns{\operatorname{ns}}

\def \int{\operatorname{int}}
\def \ns{\operatorname{ns}}
\def \f{\operatorname{fin}}
\def \inf{\operatorname{inf}}
\def \e{\operatorname{end}}
\def \ee{\operatorname{Ends}}
\def \ipc{\operatorname{IPC}}
\def \Aut{\operatorname{Aut}}

\numberwithin{equation}{section}

\allowdisplaybreaks[2]

\begin{document}

\title{Ends of groups: a nonstandard perspective}

\author{Isaac Goldbring}

\address {University of California, Los Angeles, Department of Mathematics, 520 Portola Plaza, Box 951555, Los Angeles, CA 90095-1555, USA}
\email{isaac@math.ucla.edu}
\urladdr{www.math.ucla.edu/~isaac}

\begin{abstract}
We give a nonstandard treatment of the notion of ends of proper geodesic metric spaces.  We then apply this nonstandard treatment to Cayley graphs of finitely generated groups and give nonstandard proofs of many of the fundamental results concerning ends of groups.  We end with an analogous nonstandard treatment of the ends of relatively Cayley graphs, that is Cayley graphs of cosets of finitely generated groups.
\end{abstract}

\maketitle

\section{Introduction}

Nonstandard analysis made its first serious impact on geometric group theory via the work of van den Dries and Wilkie \cite{DW} on Gromov's theorem on polynomial growth.  Indeed, the complicated limit used to form the asymptotic cone of a metric space was replaced by an ultrapower, simplifying the proof considerably.  More recently, the author used nonstandard methods to settle the local version of Hilbert's fifth problem; see \cite{Gold}.

In this paper, we treat the notion of \emph{ends of a finitely generated group} from a nonstandard perspective.  Roughly speaking, the ends of a topological space are its ``path components at infinity.''  An analsysis of the ends of the Cayley graph of a finitely generated group yields a significant amount of algebraic information about the group.  Using the language of nonstandard analysis, the aformentioned heuristic description of the ends of a proper geodesic metric space can be made precise, leading to much simpler and intuitive proofs of many of the fundamental results of the subject.

The aim of this article is two-fold:  First, we present nonstandard proofs of several basic facts about the ends of spaces and groups.  The idea is to show how the intuitive proofs of these facts can be made into rigorous nonstandard arguments, whence avoiding the lengthy and sometimes messy standard proofs.  Ideally, it is my hope that the language and techniques of nonstandard methods can provide simpler proofs of deeper theorems, e.g. Stallings Theorem on groups with infinitely many ends (see Section \ref{Cayley}), as well as lead to proofs of new results.

Secondly, we aim to show that certain notions arising naturally in the nonstandard language may lead to classical notions that have yet to be studied.  For example, we discuss a nonstandard property that a finitely generated group can possess, namely that the group have \emph{multiplicative ends}; see Section \ref{two}.  This notion suggests itself immediately once the nonstandard framework is developed, begging the question of the standard counterpart of the notion.  We present several standard characterizations of this notion, one of them being that the group is a semidirect product of a finite group by an infinite cyclic group.  It is my belief that there are other such transparent nonstandard notions whose standard counterparts may have yet to be analyzed.

I would like to thank Alberto Delgado for suggesting that I consider ends of groups in a nonstandard way.   I would also like to thank Lou van den Dries, Ilya Kapovich, and Patrick Reynolds for useful discussions concerning this work.

This research was conducted while I was a graduate student at the University of Illinois at Urbana-Champaign, where I was supported by a Dissertation Completion Fellowship. 

\

\noindent \textbf{Notations and Conventions}

\

We assume that the reader is familiar with elementary nonstandard analysis; otherwise, consult \cite{D} or \cite{He} for a friendly introduction.  Alternatively, \cite{DW} contains a short introduction to the subject for group theorists.   Here we just fix notations and terminology.
To each relevant ``basic'' set $S$ corresponds functorially
a set $S^*\supseteq S$,
the {\em nonstandard extension\/} of $S$. In particular, $\n$ and $\r$
extend
to $\n^*$ and $\r^*$,
respectively. Also, any (relevant) relation $R$ and function
$F$ on these basic sets extends functorially to a relation $R^*$ and
function
$F^*$ on the corresponding nonstandard extensions of these basic sets.
For example, the linear ordering $<$ on $\n$ extends to a
linear ordering $<^*$ on $\n^*$.  Likewise, if $G$ is a group, then the group multiplication $m:G\times G\to G$ extends to a group operation $m^*:G^*\times G^*\to G^*$.  For the sake of
readability
we only use a star in denoting the nonstandard extension of a basic set,
but drop the star when indicating the nonstandard extension
of a relation or function on these basic sets. For example,
when $x,y\in \r^*$ we write $x+y$ and $x<y$ rather than $x+^*y$ and $x<^*
y$.  The nonstandard universe is an \emph{elementary extension} of the standard universe, and when using this fact, we often say that we are using the \emph{transfer principle} or that we are \emph{arguing by transfer}.  

We remind the reader of the important notion of an \emph{internal set}.  If $S$ and its powerset $\mathcal{P}(S)$ are basic sets, then we assume that the membership relation $\in$ is a basic relation between $S$ and $\mathcal{P}(S)$.  Under this assumption, we can canonically identify $\mathcal{P}(S)^*$ with a subset of $\mathcal{P}(S^*)$.  After this identification, we call $A\subseteq S^*$ internal if it is an element of $\mathcal{P}(S)^*$.  For internal subsets of $\n^*$, there are two important principles that we use frequently throughout the paper, namely \emph{overflow} and \emph{underflow}.  Overflow states that if $A\subseteq \n^*$ is internal and contains arbitrarily large elements of $\n$, then $A$ contains an element of $\n^*\setminus \n$.  Dually, underflow states that if $A\subseteq \n^*$ is internal and contains arbitrarily small elements of $\n^*\setminus \n$, then $A$ contains an element of $\n$.  

We also assume that our nonstandard universe is $\kappa$-saturated for some sufficiently large cardinal $\kappa$ (although $\aleph_1$-saturation is probably all that is necessary).  We remind the reader that this assumption means that whenever $(A_i \ | \ i<\kappa)$ is a family of internal sets with the finite intersection property, then $\bigcap_{i<\kappa} A_i\not=\emptyset$.

Throughout this paper, $(X,x_0)$ will denote an \emph{unbounded} pointed metric space.  For any point $x\in X$ and any $R\in \r$, $B(x,R)$ will denote the \emph{closed} ball centered at $x$ with radius $R$.  For $x\in X$, we let $\mu_X(x)$ (or simply $\mu(x)$ if there is no risk of confusion) denote the \emph{monad of $x$ in $X$}, that is the set of points $y\in X^*$ for which $d(x,y)$ is infinitesimal.
We set $X_{\ns}:=\bigcup_{x\in X}\mu(x)$, the set of \emph{nearstandard elements}, that is the set of elements of $X^*$ which are infinitely close to an element of $X$.  We also set $X_{\f}:=\{x\in X^* \ | \ d(x_0,x)\in \r_{\ns}\}$, the set of elements of $X^*$  which are within a finite distance to some (equiv. any) element of $X$.  We let $X_{\inf}:=X^*\setminus X_{\f}$.

When we specialize to the case of groups, we assume that all groups are finitely generated.  To avoid trivialities, we also assume that all groups are infinite.

We always suppose $m,n,$ and $N$, sometimes subscripted, range over $\n:=\{0,1,2,\ldots\}$.

\

\section{Proper spaces and maps}

\

\noindent Recall that a metric space is said to be \emph{proper} if every closed ball is compact.  The following result is well-known, but we include a proof for the sake of completeness.

\begin{lemma}
$X$ is proper if and only if $X_{\ns}=X_{\f}$.
\end{lemma}

\begin{proof}
$(\Rightarrow)$ We always have the inclusion $X_{\ns}\subseteq X_{\f}$.  Now suppose $x\in X_{\f}$, say $d(x,x_0)\leq R$ with $R\in \r$.  Since $B(x_0,R)$ is compact, we have $B(x_0,R)^* \subseteq X_{\ns}$, whence we see that $x\in X_{\ns}$. 

$(\Leftarrow)$ Given any $x\in X$ and any $R\in \r$, we have $B(x,R)^*\subseteq X_{\f}=X_{\ns}$.  Thus, given $y\in B(x,R)^*$, there is $z\in X$ such that $y\sim z$.  Since $d(x,z)\leq R+d(y,z)$ and $d(x,z)\in \r$, it follows that $d(x,z)\leq R$ and $z\in B(x,R)_{\ns}$.  Hence, $B(x,R)$ is compact.  
\end{proof}

\

\noindent Recall that a map $f:X\to Y$ between topological spaces is said to be \emph{proper} if $f^{-1}(C)\subseteq X$ is compact for every compact $C\subseteq Y$.

\begin{lemma}
Suppose $X$ and $Y$ are proper metric spaces and $f:X\to Y$ is continuous.  Then $f$ is proper if and only if $f(X_{\inf})\subseteq Y_{\inf}$.
\end{lemma}

\begin{proof}
$(\Rightarrow)$ Suppose $f$ is proper.  Fix a basepoint $y_0$ for $Y$.  Since $B(y_0,n)$ is compact for every $n$, there is $N_n$ such that $d(y_0,f(x))>n$ for every $x\in X$ with $d(x_0,x)\geq N_n$.  Hence, for $x\in X_{\inf}$, we have $d(y_0,f(x))>n$ for every $n$, i.e. $f(x)\in Y_{\inf}$.

$(\Leftarrow)$ Suppose that $f$ is not proper.  Let $C\subseteq Y$ be compact and such that $f^{-1}(C)$ is not compact.  Since $f^{-1}(C)$ is closed, we must have that $f^{-1}(C)$ is unbounded.  Hence, by overflow, there is $x\in X_{\inf}$ satisfying $f(x)\in C^*\subseteq Y_{\ns}= Y_{\f}$.  
\end{proof}

\noindent The following special case of the previous lemma is all we will really need.  Recall that a \emph{ray in $X$} is just a continuous function $r:[0,\infty)\to X$.

\begin{cor}\label{L:properray}
Suppose $X$ is proper and $r:[0,\infty)\to X$ is a ray.  Then $r$ is proper if and only if $r(\sigma)\in X_{\inf}$ for every $\sigma \in \r^+_{\inf}$.
\end{cor}

\

\section{The Space of Ends of a Proper Geodesic Metric Space}

In this section, we assume that our unbounded pointed metric space $(X,x_0)$ is also proper.  We will use the following definition of the ends of a proper metric space.

\begin{df}
Two proper rays $r_1,r_2:[0,\infty)\to X$ are said to \emph{converge to the same end} if for every $R\in \r^{>0}$, there exists $N$ such that $r_1[N,\infty)$ and $r_2[N,\infty)$ are contained in the same path component of $X\setminus B(x_0,R)$.  This defines an equivalence relation on the set of proper rays; the equivalence class of $r$ will be denoted by $\e(r)$.  The set of equivalence classes will be denoted by $\ee(X)$.  
\end{df}

Before we give a nonstandard characterization of two proper rays having the same end, we must introduce and analyze a few nonstandard notions.

\begin{df} For $x,y\in X^*$, we write $x\propto y$ if there is $\alpha \in C([0,1],X)^*$ such that $\alpha(0)=x$, $\alpha(1)=y$, and $\alpha(t)\in X_{\inf}$ for all $t\in [0,1]^*$.  
\end{df}

Heuristically, one should think of the relation $x\propto y$ as saying that $x$ and $y$ are in the same ``path component at infinity'', for there is an \emph{internal} path connecting $x$ and $y$ which is contained in the infinite portion of the space.

\begin{df}
For $x,y\in X$ and $R\in \r^{>0}$, we write $x\propto_R y$ if there is $\alpha \in C([0,1],X)$ such that $\alpha(0)=x$, $\alpha(1)=y$, and $\alpha(t)\in X\setminus B(x_0,R)$ for all $t\in [0,1]$.  Note that $\propto_R$ is an equivalence relation on $X$ for each $R\in \r^{>0}$.  We will also use $\propto_\sigma$ for $\sigma\in\r^*\setminus \r$, which is the internal relation on $X^*$ given by $x\propto_\sigma y$ if and only if there is $\alpha \in C([0,1],X)^*$ such that $\alpha(0)=x$, $\alpha(1)=y$, and $\alpha(t)\in X^*\setminus B(x_0,\sigma)$ for all $t\in [0,1]^*$.
\end{df}

\begin{rmk}
Suppsoe $x,y\in X^*$.  If $\sigma\in \r^*\setminus \r$ and $x\propto_\sigma y$, then $x\propto y$.  Conversely, if $x\propto y$, then, by underflow, there exists $\nu\in \n^*\setminus \n$ such that $x\propto_\nu y$; of course such a $\nu$ depends on $x$ and $y$.
\end{rmk}

\begin{rmk}\label{L:asymprmks}
The choice of $[0,1]^*$ in the above definitions is purely arbitrary.  In fact, let $\operatorname{Path}(X)$ denote the set of paths in $X$, that is $\alpha\in \operatorname{Path}(X)$ if and only if there are $r,s\in \r$ such that $\alpha:[r,s]\to X$ is continuous.  Note that any $\alpha \in \operatorname{Path}(X)$ has a reparameterization in $C([0,1],X)$.  Hence, by transfer, if there is $\alpha\in \operatorname{Path}(X)^*$, say $\alpha:[\sigma,\tau]\to X^*$, such that $\alpha(\sigma)=x$, $\alpha(\tau)=y$, and $\alpha(t)\in X_{\inf}$ for all $t\in[\sigma,\tau]$, then $x\propto y$; likewise for the notion of $\propto_\nu$.  (Here $\sigma$ and $\tau$ are in $\r^*$ and $[\sigma,\tau]$ denotes the interval determined by $\sigma$ and $\tau$ in $\r^*$.)  It follows that $\propto$ is an equivalence relation on $X_{\inf}$.
\end{rmk}

\noindent In proper \emph{geodesic} spaces, we can find a ``discrete'' formulation of $\propto$.  Recall that $X$ is a geodesic metric space if for any $x,y\in X$, there is an isometric embedding $\alpha:[0,r]\to X$ such that $\alpha(0)=x$ and $\alpha(r)=y$, where $r:=d(x,y)$.

\begin{lemma}\label{L:asympgeo}
Suppose $X$ is a proper geodesic space.  Fix $x,y\in X^*$.  Then the following are equivalent:
\begin{enumerate}
\item $x\propto y$;
\item for every $\epsilon \in (\r^{>0})^*$, there is a hyperfinite sequence $a_0,\ldots,a_\nu$ in $X_{\inf}$ such that $a_0=x$, $a_\nu=y$, and $d(a_i,a_{i+1})<\epsilon$ for each $i<\nu$;
\item there is a hyperfinite sequence $a_0,\ldots,a_\nu$ in $X_{\inf}$ such that $a_0=x$, $a_\nu=y$, and $d(a_i,a_{i+1})\in \r_{\f}$ for each $i<\nu$.
\end{enumerate}
\end{lemma}

\begin{proof}
$(1)\Rightarrow (2)$:  Fix $\epsilon\in (\r^{>0})^*$.  Fix $\alpha \in C([0,1],X)^*$ witnessing that $x\propto y$.  Since $\alpha$ is internally uniformly continuous, there is $\nu\in \n^*\setminus \n$ such that for all $t,t'\in [0,1]^*$, if $|t-t'|\leq \frac{1}{\nu}$, then $d(\alpha(t),\alpha(t'))<\epsilon$.  The desired sequence is then given by $a_i:=\alpha(\frac{i}{\nu})$.

$(2)\Rightarrow (3)$ is trivial.

$(3)\Rightarrow (1)$:  Let the hyperfinite sequence $a_0,\ldots,a_\nu$ be as guaranteed to exist by $(3)$.  For each $i<\nu$, let $[a_i,a_{i+1}]$ denote an internal geodesic segment connecting $a_i$ and $a_{i+1}$.  Since $d(a_i,a_{i+1})\in \r_{\f}$, these segments are contained entirely in $X_{\inf}$.  Concatenating these segments and applying Remark \ref{L:asymprmks}, we see that $x\propto y$. 
\end{proof}

\noindent We are now prepared to give a nonstandard characterization of two proper rays having the same end.

\begin{lemma}\label{L:nonstend}
Suppose $r_1,r_2:[0,\infty)\to X$ are proper rays.  Then the following are equivalent:
\begin{enumerate}
\item $\e(r_1)=\e(r_2)$;
\item for \emph{all} $\sigma, \tau \in \r^{>0}_{\inf}$, $r_1(\sigma)\propto r_2(\tau)$;
\item for \emph{some} $\sigma, \tau \in \r^{>0}_{\inf}$, $r_1(\sigma)\propto r_2(\tau)$.
\end{enumerate}
\end{lemma}

\begin{proof}
$(1) \Rightarrow(2)$:  Suppose $\e(r_1)=\e(r_2)$ and let $\sigma, \tau \in \r^{>0}_{\inf}$.  Given $n$, there is $N$ such that for all $s,t\in \r^{>0}$ with $s,t\geq N$, we have $r_1(s)\propto_n r_2(t)$.  Consider the internal set $$A:=\{\nu\in \n^* \ | \ r_1(\sigma)\propto_\nu r_2(\tau)\}.$$  By the transfer principle, $\n\subseteq A$.  Thus, by overflow, we have $\nu\in \n^*\setminus \n$ with $r_1(\sigma)\propto_\nu r_2(\tau)$, yielding that $r_1(\sigma)\propto r_2(\tau)$.

$(2) \Rightarrow(1)$:  Suppose $\e(r_1)\not= \e(r_2)$.  Then there is some $R\in \r^{>0}$ such that $r_1[N,\infty)$ and $r_2[N,\infty)$ do not lie in the same path component of $X\setminus B(x_0,R)$ for every $N$; that is, for every $N$, there are $s,t\in \r^{>0}$ with $s,t\geq N$ such that $r_1(s)\not \propto_R r_2(t)$. For each $N$, consider the internal set
$$B_N:=\{(s,t)\in \r^*\times \r^* \ | \ s,t\geq N \text{ and } r_1(s) \not \propto_R r_2(t)\}.$$  By assumption, each $B_N$ is nonempty.  By saturation, there exists $(\sigma,\tau) \in \bigcap \{B_N \ | \ N\in \n\}$.  Then $\sigma,\tau \in \r^{>0}_{\inf}$ and $r_1(\sigma)\not \propto_R r_2(\tau)$, which implies that $r_1(\sigma)\not \propto r_2(\tau)$.

$(2)\Rightarrow (3)$ is trivial.

$(3)\Rightarrow (2)$:  Suppose $\sigma, \tau \in \r^{>0}_{\inf}$ are such that $r_1(\sigma)\propto r_2(\tau)$ and let $\sigma', \tau'\in \r^{>0}_{\inf}$ be arbitrary.  Then (2) follows from the fact that $r_1(\sigma)\propto r_1(\sigma')$ and $r_2(\tau)\propto r_2(\tau')$, which in turn follows from Lemma \ref{L:properray} and Remark \ref{L:asymprmks} (2).
\end{proof}

\noindent The following lemma combines Lemmas \ref{L:asympgeo} and \ref{L:nonstend}.

\begin{lemma}\label{L:nonstendgeo}
Suppose $X$ is a proper \emph{geodesic} space and $r_1,r_2:[0,\infty)\to X$ are proper rays.  Then the following are equivalent:
\begin{enumerate}
\item $\e(r_1)=\e(r_2)$;
\item For all (equiv. for some) $\sigma, \tau \in \r^{>0}_{\inf}$, $r_1(\sigma)\propto r_2(\tau)$;
\item For all (equiv. for some) $\sigma, \tau \in \r^{>0}_{\inf}$ and every $\epsilon\in (\r^*)^{>0}$, there is a hyperfinite sequence $a_0,\ldots,a_\nu$ in $X_{\inf}$ such that $a_0=r_1(\sigma)$, $a_\nu=r_2(\tau)$ and $d(a_i,a_{i+1})<\epsilon$ for each $i<\nu$;
\item For all (equiv. for some) $\sigma, \tau \in \r^{>0}_{\inf}$, there is a hyperfinite sequence $a_0,\ldots,a_\nu$ in $X_{\inf}$ such that $a_0=r_1(\sigma)$, $a_\nu=r_2(\tau)$ and $d(a_i,a_{i+1})\in \r_{\f}$ for each $i<\nu$.
\end{enumerate}
\end{lemma}

\noindent For $x\in X_{\inf}$, let $[x]$ denote its equivalence class under $\propto$ and refer to $[x]$ as the \emph{infinite path component of $x$}.  We denote the set of infinite path components of $X$ by $$\ipc(X):=\{[x] \ | \ x\in X_{\inf}\}.$$  Fix $\sigma \in \r^{>0}_{\inf}$.  Then Lemma \ref{L:nonstend} allows us to define a map $$\Theta:\ee(X)\to \operatorname{IPC}(X), \quad \Theta(\e(r))=[r(\sigma)].$$  Lemma \ref{L:nonstend} further implies that $\Theta$ is injective and independent of the choice of $\sigma$.

\begin{lemma}\label{L:IPC}
Suppose $X$ is a proper \emph{geodesic} space.  Then $\Theta$ is a bijection.  
\end{lemma}   

\begin{proof}
Let $x\in X_{\inf}$ and let $\sigma:=d(x,x_0)\in \r^{>0}_{\inf}$.  Let $\hat{r}:[0,\sigma]\to X^*$ be an internal geodesic connecting $x_0$ and $x$.  Note that $\hat{r}(t)\in X_{\f}=X_{\ns}$ for $t\in \r^{>0}_{\f}$.  We may thus define $r:[0,\infty)\to X$ by $r(t):=\st(\hat{r}(t))$.  Note that $r$ is a geodesic ray.  Indeed, for $t,t'\in [0,\infty)$, we have
$$d(r(t),r(t'))=d(\st(\hat{r}(t)),\st(\hat{r}(t')))=\st(d(\hat{r}(t),\hat{r}(t')))=\st(|t-t'|)=|t-t'|.$$ To finish the proof of the lemma, it suffices to show that $r(\sigma)\propto x$, as then $\Theta(\e(r))=[x]$.  Fix $\epsilon \in \r^{>0}$.  Then the set $$\{\nu\in \n^* \ | \ \nu\leq \sigma \wedge d(\hat{r}(\nu),r(\nu))<\epsilon\}$$ is internal and contains all of $\n$.  By overflow, we must have $\nu\in \n^*\setminus \n$ such that $\nu\leq \sigma$ and $d(r(\nu),\hat{r}(\nu))<\epsilon$.  Connecting $r(\nu)$ and $\hat{r}(\nu)$ by an internal geodesic, we see that $r(\nu)\propto \hat{r}(\nu)$.  Since $r(\sigma)\propto r(\nu)$ and $\hat{r}(\nu)\propto \hat{r}(\sigma)=x$, we are finished.  
\end{proof}

\begin{rmk}
The above lemma makes it immediately clear that the proper geodesic space $\r^n$, equipped with its usual metric, has two ends if $n=1$ and one end if $n\geq 2$.
\end{rmk}

\noindent \textbf{Notation:}  Let $\G_{x_0}(X)$ denote the set of geodesic rays in $X$ emanating from $x_0$.

\begin{cor}\label{T:geosurj}
The map $r\mapsto \e(r):\G_{x_0}(X)\to \ee(X)$ is surjective.
\end{cor}

\begin{proof}
Immediate from the proof of Lemma \ref{L:IPC}.
\end{proof}

\noindent A useful property of a space with finitely many ends is that one can ``separate" the ends with a ball centered around $x_0$ of finite radius.  This may not be possible for a space with infinitely many ends.  However, we can separate the ends with a ball centered at $x_0$ of \emph{hyperfinite radius}.

\begin{lemma}
Suppose that $X$ has infinitely many ends.  Let $\{r_i \ | \ i\in I\}\subseteq \G_{x_0}(X)$ be distinct such that $\{\e(r_i) \ | \ i\in I\}$ enumerates the ends of $X$.  Then for every $\sigma\in \n^*\setminus \n$ and for all distinct $i,j\in I$, we have $r_i(\sigma)\not\propto r_j(\sigma)$.
\end{lemma}

\begin{proof}
Immediate from Lemma \ref{L:nonstend}.
\end{proof}

\noindent If $X$ is a proper geodesic space, then $\ee(X)$ can be topologized in the following manner.  Fix $r\in \G_{x_0}(X)$.  For $n>0$, let $\tilde{V_n}(r)$ be the set of $r'\in \G_{x_0}(X)$ such that $r'(m)\propto_n r(m)$ for all (equiv. some) $m>n$.  Let $V_n(r):=\{\e(r') \ | \ r'\in \tilde{V_n}(r)\}$.  Then the sets $V_n(r)$ form a neighborhood basis of $\e(r)$ in $\ee(X)$.

We now give a nonstandard description of the topology on $\ee(X)$ by giving a description of the monad system of $\ee(X)$.  By Corollary \ref{T:geosurj}, we can think of $\ee(X)$ as $\G_{x_0}(X)$ modulo the equivalence relation of two rays having the same end.  By the Transfer Principle, $\G_{x_0}(X)^*$ is the set of internally geodesic rays in $X^*$ emanating from $x_0$ and $\ee(X)^*$ is the quotient of $\G_{x_0}(X)^*$ modulo the equivalence relation which is the extension of the equivalence relation of two rays having the same end.  One should note that if $r\in \G_{x_0}(X)^*$, then $r(\sigma)\in X_{\inf}$ for $\sigma \in \r^{>0}_{\inf}$.

\begin{lemma}\label{L:endmonad}
For $r\in \G_{x_0}(X)$ and $r' \in \G_{x_0}(X)^*$, we have $\e(r')\in \mu(\e(r))$ if and only if $r'(\sigma)\propto r(\sigma)$ for some (equivalently, all) $\sigma \in \r^{>0}_{\inf}$.
\end{lemma}

\begin{proof}
$(\Rightarrow)$ Suppose $\e(r')\in \mu(\e(r))$.  Consider the internal set
$$A:=\{\nu\in \n^* \ | \ (\forall \sigma \in \r^*)(\sigma>\nu \to r(\sigma)\propto_\nu r'(\sigma)\}.$$  Since $r'\in \tilde{V_n}(r)^*$ for each $n$, we have $\n \subseteq A$.  By overflow, there is $\nu \in \n^*\setminus \n$ such that $\nu \in A$.  Hence, if $\sigma \in \r^*$ is such that $\sigma >\nu$, then $r(\sigma) \propto r'(\sigma)$.

$(\Leftarrow)$ Suppose $\sigma \in \r^{>0}_{\inf}$ is such that $r'(\sigma)\propto r(\sigma)$.  Fix $n>0$.  We want to show that $\e(r')\in V_n(r)^*$.  For $m>n$, consider the internal path connecting $r'(m)$ and $r(m)$ obtained by first connecting $r'(m)$ and $r'(\sigma)$ using $r'$, then connecting $r'(\sigma)$ and $r(\sigma)$ with an internal path contained entirely in $X_{\inf}$, then finally connecting $r(\sigma)$ and $r(m)$ using $r$.  This internal path lies entirely in $X^*\setminus B(x_0,n)^*$, so $\e(r')\in V_n(r)^*$. 
\end{proof}

\begin{cor}
$\ee(X)$ is a compact hausdorff space.
\end{cor}

\begin{proof}
Lemmas \ref{L:nonstend} and \ref{L:endmonad} make it clear that any two distinct monads are disjoint, whence $\ee(X)$ is hausdorff.  Now suppose $r'\in \G_{x_0}(X)^*$.  To show that $\ee(X)$ is compact, we need to find $r\in \G_{x_0}(X)$ such that $\e(r')\in \mu(\e(r))$.  The desired geodesic ray is obtained by defining $r(t):=\st(r'(t))$; the details are identical to those in the proof of Lemma \ref{L:IPC}.
\end{proof}

Equip $\ipc(X)$ with the unique topology which makes $\Theta$ a homeomorphism.  The next lemma gives a more concrete description of this topology on $\ipc(X)$.

\begin{lemma}
Fix $[x]\in \ipc(X)$.  Let $V_n([x]):=\{[x']\ | \ x'\propto_n x\}$.  Then the family of sets $V_n([x])$ form a basis of neighborhood of $[x]$ in $\ipc(X)$. 
\end{lemma}

\begin{proof}
Fix $r\in \G_{x_0}(X)$ such that $\Theta(\e(r))=[x]$.  We will show that $\Theta(V_n(r))=V_n([x])$.  The fact that $\Theta(V_n(r))\subseteq V_n([x])$ follows immediately from the definitions and the transfer principle.  Now suppose that $[x']\in V_n([x])$.  Choose $r'\in \G_{x_0}(X)$ such that $\Theta(\e(r'))=[x']$.  Then $r'(\sigma)\propto_n r(\sigma)$, whence it follows that $r'\in \tilde{V_n}(r)$ and hence $[x']\in \Theta(V_n(r))$.  
\end{proof}

\section{Ends and Quasi-Isometries}

In this section, we use our nonstandard description of ends to give a proof of the fact that quasi-isometries between two proper geodesic spaces induce homemorphisms on the corresponding end spaces.  We begin by defining quasi-isometries and proving a few facts concerning quasi-isometries in the nonstandard framework.

\begin{df}
Suppose that $(X,d_X)$ and $(Y,d_Y)$ are metric spaces.  For $\lambda\in \r^{\geq 1}$ and $\epsilon\in \r^{>0}$, a (not necessarily continuous) function $f:X\to Y$ is a \textbf{$(\lambda,\epsilon)$-quasi-isometric embedding} if, for all $x,x'\in X$, we have
$$\frac{1}{\lambda}d_Y(f(x),f(x'))-\epsilon \leq d_X(x,x')\leq \lambda d_Y(f(x),f(x'))+\epsilon.$$  If $f:X\to Y$ is a $(\lambda,\epsilon)$-quasi-isometric embedding, we call $f$ a \textbf{$(\lambda,\epsilon)$-quasi-isometry} if there is $C\in \r^{>0}$ such that the $C$-neighborhood of $f(X)$ equals $Y$.  We say that $f:X\to Y$ is a \textbf{quasi-isometric embedding} if it is a $(\lambda,\epsilon)$-quasi-isometric embedding for some $\lambda$ and $\epsilon$.  Similarly, $f:X\to Y$ is a \textbf{quasi-isometry} if it is a $(\lambda,\epsilon)$-quasi-isometry for some $\lambda$ and $\epsilon$.  It can be shown that if $f:X\to Y$ is a quasi-isometry, then there is a \textbf{quasi-inverse for $f$}, which is a quasi-isometry $g:Y\to X$ for which there is $K\in \r^{>0}$ such that, for all $x\in X$ and $y\in Y$, we have $d_X(g(f(x)),x))\leq K$ and $d_Y(f(g(y)),y)\leq K$. 
\end{df}

\begin{lemma}\label{L:qinonst}
Suppose $X$ and $Y$ are proper geodesic spaces and $f:X\to Y$ is a quasi-isometric embedding.  Then:
\begin{enumerate}
\item For all $x,x'\in X^*$, $d(x,x')\in \r_{\f}$ if and only if $d(f(x),f(x'))\in \r_{\f}$.
\item If $x,x'\in X_{\inf}$ are such that $x\propto x'$, then $f(x)\propto f(x')$.  Moreover, if $f$ is a quasi-isometry, then for all $x,x'\in X_{\inf}$, $x\propto x'$ if and only if $f(x)\propto f(x')$.
\end{enumerate}
\end{lemma}

\begin{proof}
$(1)$ follows immediately from the definition of a quasi-isometric embedding.  
For (2), fix $x,x'\in X_{\inf}$ such that $x\propto x'$.  By Lemma \ref{L:asympgeo}, there is a hyperfinite sequence $a_0,\ldots,a_\nu$ from $X_{\inf}$ such that $a_0=x$, $a_\nu=y$, and $d(a_i,a_{i+1})\in \r_{\f}$ for all $i<\nu$.  Then by (1), $f(a_0),\ldots,f(a_\nu)$ is a hyperfinite sequence from $Y_{\inf}$ such that $f(a_0)=f(x)$, $f(a_\nu)=f(x')$, and $d(f(a_i),f(a_{i+1}))\in \r_{\f}$ for all $i<\nu$, whence $f(x)\propto f(x')$ by Lemma \ref{L:asympgeo}.  Now suppose that $f$ is a quasi-isometry and $f(x)\propto f(x')$.  Let $g:Y\to X$ be a quasi-inverse for $f$.  By the first part of (2), we have $g(f(x))\propto g(f(x'))$.  Since $g(f(x))$ and $x$ are within a finite distance from each other, we have $g(f(x))\propto x$; likewise $g(f(x'))\propto x'$, whence we have $x\propto x'$.
\end{proof}

\noindent We now have the following well-known standard corollary.

\begin{cor}\label{L:qistand}
Suppose $X$ and $Y$ are proper geodesic spaces and $f:X\to Y$ is a quasi-isometric embedding.  Then for every $n$, there is $m$ such that for all $x,x'\in X$, if $x\propto_m x'$, then $f(x)\propto_n f(x')$.
\end{cor}

\begin{proof}
Suppose, towards a contradiction, that there is $n$ such that for every $m$, there are $x_m,x_m'\in X$ such that $x_m\propto_m x_m'$ and $f(x_m)\not\propto_nf(x_m')$.  Then by saturation, there is $x,x'\in X^*$ such that $x\propto_m x'$ for all $m$ and yet $f(x)\not\propto_n f(x')$.  Set $$A_m:=\{\alpha\in C([0,1])^* \ | \ \alpha(0)=x, \alpha(1)=x', \forall t\in [0,1]^* (d(x_0, \alpha(t))\geq m)\}.$$  By choice of $x$ and $x'$, $A_m$ has the finite intersection property, so by saturation, there is $\alpha \in \bigcap_m A_m$.  Then $\alpha$ witnesses that $x\propto x'$.  However, $f(x)\not\propto f(x')$, contradicting the previous lemma.   
\end{proof}

\noindent The composition of two quasi-isometries is once again a quasi-isometry.  In order to form the quasi-isometry group of $X$, we first need to identify two quasi-isometries which are a bounded distance away from each other.  More precisely, for two functions $f,g:X\to X$, say that $f$ and $g$ are \emph{equivalent} if there is $R\in \r$ such that $d(f(x),g(x))\leq R$ for all $x\in X$.  Then $\operatorname{QI}(X)$, the \emph{quasi-isometry group of $X$}, will denote the set of equivalence classes of quasi-isometries of $X$ equipped with the operation induced by composition of quasi-isometries.

\begin{lemma}\label{L:qicont}
Suppose that $X$ and $Y$ are proper geodesic spaces.  Then every quasi-isometry $f:X\to Y$ induces a homeomorphism $f_e:\ipc(X)\to \ipc(Y)$.  The map $$f\mapsto f_e:\operatorname{QI}(X)\to \operatorname{Homeo}(\ipc(X))$$ is a group homomorphism.
\end{lemma}

\begin{proof}
By Lemma \ref{L:qinonst}, we can define $f_e([x]):=[f(x)]$.  We claim that $f_e$ is continuous.  Fix $[x]\in \ipc(X)$ and $n$.  The transfer principle applied to Corollary \ref{L:qistand} shows that there is $m$ such that $f_e(V_m([x]))\subseteq V_n([f(x)])$, from which the continuity of $f_e$ follows.  Now suppose that $g:Y\to X$ is also a quasi-isometry.  Then $$g_e(f_e([x]))=g_e([f(x)])=[(g(f(x))]=(g\circ f)_e([x]).$$  If $g$ happened to be a quasi-inverse to $f$, then the fact that $d(g(f(x)),x)\in \r_{\f}$ shows that $g(f(x))\propto x$, whence $g_e(f_e([x]))=[x]$ and $g_e$ is the inverse to $f_e$.
\end{proof}

\begin{cor}[\cite{BH}, Proposition 8.29]
Suppose that $X$ and $Y$ are proper geodesic spaces.  Then every quasi-isometry $f:X\to Y$ induces a homeomorphism $f_e:\ee(X)\to \ee(Y)$.  The map $$f\mapsto f_e:\operatorname{QI}(X)\to \operatorname{Homeo}(\ee(X))$$ is a group homomorphism.
\end{cor}

\begin{proof}
This is immediate from the previous lemma and Lemma \ref{L:IPC}.  The constructions involved show that $f_e(\e(r))$ is the end associated to $f(r(\sigma))$ for any $\sigma \in \r^{>0}_{\inf}$. 
\end{proof}

\section{Application to Cayley Graphs of Finitely Generated Groups}\label{Cayley}

\noindent In this section, we specialize to the case that $X$ is the metric space associated to the Cayley graph of a finitely generated group.  We first consider the more general context of a locally finite combinatorial graph.

Suppose that $(\V,\E)$ is a \emph{locally finite} combinatorial graph, that is a combinatorial graph for which every vertex has only finitely many edges emanating from it.  We can turn $(\V,\E)$ into a metric space $X:=X(\V,\E)$ by identifying each edge with an isometric copy of the interval $[0,1]$ and then declaring, for $x,y\in X$, $d(x,y)$ to be the infimum of the lengths of paths from $x$ to $y$; see \cite{BH} for more details.  In this way, $X$ becomes a proper geodesic space.  Let us fix a basepoint $x_0$ of $X$, which we assume to be an element of $\V$.  Let us agree to write $\V_{\f}$ for $X_{\f}\cap \V^*$ and $\V_{\inf}$ for $\V^*\setminus \V_{\f}$.  Since $\V\cap B(x_0,n)$ is finite for any $n$, we see that $\V_{\f}=\V$ (whence $\V_{\inf}=\V^*\setminus \V$).  Also, by the Transfer Principle, for every $x\in X^*$, there is $v\in \V^*$ with $d(x,v)\leq 1$, whence every infinite path component has a representative from $\V_{\inf}$, that is $$\operatorname{IPC}(X):=\{[v] \ | \ v\in \V_{\inf}\}.$$

\begin{lemma}\label{L:asympgraph}
For $v,v'\in \V_{\inf}$, we have $v\propto v'$ if and only if there is a hyperfinite sequence $g_0,\ldots,g_\nu$ from $V_{\inf}$ such that $g_0=v$, $g_\nu=v'$, and $(g_i,g_{i+1})\in \E^*$ for all $i<\nu$.
\end{lemma}

\begin{proof}
The backward direction is immediate from the direction $(3)\Rightarrow (1)$ of Lemma \ref{L:asympgeo}.  For the proof of the forward direction, suppose $v\propto v'$.  By $(1)\Rightarrow (2)$ of Lemma \ref{L:asympgeo}, we have a hyperfinite sequence $a_0,\ldots,a_\eta$ in $X_{\inf}$ such that $a_0=v$, $a_\eta=v'$, and $d(a_i,a_{i+1})<\frac{1}{2}$ for all $i<\eta$.  Now define the internal set $R\subseteq \n^*\times \V^*$ by $(i,x)\in R$ if and only if $a_i$ and $a_{i+1}$ lie on the interiors of distinct edges (so in particular, $a_i,a_{i+1}\notin \V^*$) and $x$ is the unique vertex lying in between $a_i$ and $a_{i+1}$.  Let $\pi_1:\n^*\times \V^*\to \n^*$ and $\pi_2:\n^*\times \V^*\to \V^*$ denote the projections onto $\n^*$ and $\V^*$ respectively.  Note that $\pi_2(R)\subseteq \V_{\inf}$.  For $j\in \pi_1(R)$, let $b_j\in \V^*$ be such that $(j,b_j)\in R$.  Let $\eta':=\eta+|\pi_1(R)|$ and define a hyperfinite sequence $c_0,\ldots,c_{\eta'}$ from $X^*\times \{0,1\}$ by internal recursion as follows.  Let $c_0=(v,0)$.  Suppose that $i>0$ and that $c_{i-1}$ has been defined.  Then define $c_i$ by
\[
	c_i=
	\begin{cases}
	(a_{j+1},0)	&\text{if $c_{i-1}=(a_j,0)$ and $j\notin \pi_1(R)$}\\
	(b_j,1)		&\text{if $c_{i-1}=(a_j,0)$ and $j\in \pi_2(R)$}\\
	(a_{j+1},0)	&\text{if $c_{i-1}=(b_j,1)$}.
	\end{cases}  
\]
The idea here is to insert vertices into the original sequence which lie in between consecutive elements of the original sequence.  We use $0$ and $1$ as labels to distinguish original members of the sequence from newly added members of the sequence.  Let $\pi:X^*\times \{0,1\}\to X^*$ be the projection map.  Let $$\eta'':=|\pi(X^*\times \{0,1\})\cap \V^*|.$$  Define the hyperfinite sequence $d_0,\ldots,d_{\eta''}$ from $\V^*$ by internal recursion as follows.  Let $d_0=v=\pi(c_0)$.  Now suppose that $i>0$ and $d_{i-1}=\pi(c_j)$.  Then define $d_i=\pi(c_k)$, where $k\in \{1,\ldots,\eta'\}$ is minimal satisfying the requirements that $k>j$ and $\pi(c_k)\in \V^*$.  Note that successive $d_i$'s are either equal or a distance $1$ apart.  Let $\nu:=|\{d_i \ | \ i\leq \eta''\}|$ and define the hyperfinite sequence $g_0,\ldots,g_\nu$ by internal recursion as follows.  Let $g_0=v=d_0$.  Now suppose that $i>0$ and $g_{i-1}=d_j$.  Then define $g_i=d_k$ where $k\in \{1,\ldots,\eta''\}$ is minimal satisfying $k>j$ and $d_k\not=d_j$.  This sequence is as desired.
\end{proof}

\noindent Fix a group $G$, which by the convention established in the introduction is assumed to be finitely generated and infinite.  Fix a finite generating set $S$ for $G$.  We let $\operatorname{Cay}(G,S)$, the \textbf{Cayley graph of $G$ with respect to the generating set $S$}, be the locally finite graph with $\V=G$ and edge relation given by $(g,h)\in \E$ if and only if there is $s\in S^{\pm 1}$ such that $h=gs$.  (Here, $S^{\pm 1}=S\cup S^{-1}$, where $S^{-1}:=\{s^{-1} \ | \ s\in S\}$.)  We let $X$ denote the metric space associated to $\operatorname{Cay}(G,S)$.  We take $x_0=1$ as our basepoint in $X$.  If $S'$ is also a finite generating set for $G$ and $X'$ is the metric space associated to $\operatorname{Cay}(G,S')$, then $X'$ is quasi-isometric to $X$ (see \cite{BH}, Chapter I.8, Example 8.17(3)), whence $\ee(X)$ and $\ee(X')$ are homeomorphic by Lemma \ref{L:qicont}.  Hence, defining $\ee(G):=\ee(X)$ gives us a space which is uniquely determined up to homeomorphism.

For $g\in G$, let $|g|:=d(1,g)$.  We have $G_{\f}=G=\{g\in G^* \ | \ |g|\in \r_{\f}\}$, $G_{\inf}=G^*\setminus G$, and $\operatorname{IPC}(X)=\{[g] \ | \ g\in G_{\inf}\}.$

The following group-theoretic interpretation of when $g\propto g'$ follows immediately from Lemma \ref{L:asympgraph}.

\begin{lemma}\label{L:asympcay}
For $g,g'\in G_{\inf}$, we have $g\propto g'$ if and only if there is a hyperfinite sequence $s_0,\ldots,s_\eta \in S^{\pm 1}$ such that $gs_0\cdots s_\eta=g'$ and $gs_0\cdots s_i\in G_{\inf}$ for all $i\in \{1,\ldots,\eta\}$.
\end{lemma}

\noindent The action of $G$ on itself by left multiplication extends to an isometry of $X$ (as it preserves the relation $\mathcal{E}$), whence Lemma \ref{L:qicont} yields a group morphism $$g\mapsto ([x]\mapsto [gx]):G\to \operatorname{Homeo}(\ipc(X)).$$  Let $H$ be the kernel of this group morphism, so for $h\in H$ and $x\in G_{\inf}$, we have $hx\propto x$.  We will call $H$ the \textbf{end stabilizer} of $G$.  By Lemma \ref{L:qinonst}, $H$ is independent of the choice of $S$.  Under the identification between $\ipc(X)$ and $\ee(X)$, this morphism becomes $$g\mapsto (\e(r)\mapsto \e(g\cdot r)):G\to \operatorname{Homeo}(\ee(X)).$$  Then for $h\in H$ and $\e(r)\in \ee(X)$, we have $\e(h\cdot r)=\e(r)$.

\

We now use everything that we have developed thus far to give a nice nonstandard proof of one of the fundamental theorems of the subject.

\begin{thm}[Hopf, \cite{Hopf}]\label{T:numberofends}
Suppose that $G$ has finitely many ends.  Then $G$ has at most two ends.
\end{thm}

\begin{proof}
Suppose that $G$ has finitely many ends but, towards a contradiction, at least 3 ends, say  $e_1,e_2,e_3$.  For $i=1,2$, let $r_i\in \G_1(X)$ be such that $\e(r_i)=e_i$.  Since $H$ has finite index in $G$, we see that there is a fixed constant $K$ such that every element of $G$ is within $K$ of an element of $H$.  Thus there is a proper ray $r:[0,\infty)\to X$ with $\e(r)=e_3$ and such that $|r(n)|\geq n$ and $r(n)\in H$ for each $n$.  We will need the following claim.

\

\noindent \textbf{Claim}:  There are $\beta, \nu \in \n^*\setminus \n$ such that $r(\beta)r_i(\nu) \in G_{\inf}$ and $r(\beta)r_i(\nu) \propto r_i(\nu)$ for $i=1,2$.

\

\noindent The reason that the claim is not trivially true by overflow is that the relation $\propto$ is \emph{external}, that is, not internal.  Fix $\gamma\in \r^{>0}_{\inf}$.  For each $n>0$, apply the transfer principle to the the fact that $$\e(r(n)\cdot r_1)=\e(r_1) \text { and }\e(r(n)\cdot r_2)=\e(r_2)$$ to obtain $\nu \in \n^*$ with $\nu>\gamma$ satisfying $$r(n)\cdot r_1(\nu)\propto_\gamma r_1(\nu) \text{ and }r(n)\cdot r_2(\nu)\propto_\gamma r_2(\nu).$$  Now we can apply overflow to obtain $\beta \in \n^*\setminus \n$ such that there is $\nu \in \n^*$ with $\nu>\gamma$ such that $r(\beta)\cdot r_1(\nu)\propto_\gamma r_1(\nu)$ and $r(\beta)\cdot r_2(\nu)\propto_\gamma r_2(\nu)$, proving the claim.

\

\noindent Fix $\beta$ and $\nu$ as in the Claim.  Let $h:=r(\beta)$ and $x_i:=r_i(\nu)$, $i=1,2$.  Note that we can write $x_1=s_1\cdots s_\nu$, where $s_\eta \in S^{\pm 1}$ and $|s_1\cdots s_\eta|=\eta$ for all $\eta\in \{1,\ldots,\nu\}$.  Likewise, $x_2=t_1\cdots t_{\nu}$, where $t_\eta\in S^{\pm 1}$ and $|t_1\cdots t_\eta|=\eta$ for all $\eta\in \{1,\ldots,\nu\}$.  Since $hx_1\propto x_1\not\propto h$, Lemma \ref{L:asympcay} implies that $hs_1\cdots s_\eta \in G$ for some $\eta<\nu$.  Likewise, $ht_1\cdots t_\zeta\in G$ for some $\zeta<\nu$.  Since $h\in G_{\inf}$, we must have $s_1\cdots s_\eta,t_1 \cdots t_\zeta \in G_{\inf}$, whence $\eta,\zeta\in \n^*\setminus\n$.  Since $s_\eta^{-1}\cdots s_1^{-1}t_1\cdots t_\zeta\in G$, it  follows that $s_1\cdots s_\eta \propto t_1\cdots t_\zeta$, and since $x_1 \propto s_1\cdots s_\eta$ and $x_2 \propto t_1\cdots t_\zeta$, we get $x_1\propto x_2$, a contradiction.

\end{proof}

\begin{rmk} There are finitely generated groups with exactly one end.  Indeed, the Cayley graph of $\z \oplus \z$ is quasi-isometric to $\r^2$, whence $\z\oplus \z$ has one end.  (We will consider a generalization of this fact in Lemma \ref{L:prodinf}.)  Note that the Cayley graph of $\z$ is quasi-isometric to $\r$, whence $\z$ has two ends.  In fact, $G$ has two ends if and only if it is \emph{virtually} $\z$, that is if and only if it has a subgroup of finite index which is isomorphic to $\z$.  The ``if'' direction of this result follows from the fact that if $G$ is a finitely generated group with finite generating set $S$ and $G'$ is a finitely generated subgroup of finite index in $G$ with generating set $S'\subseteq S$, then the natural inclusion $\operatorname{Cay}(G',S')\hookrightarrow \operatorname{Cay}(G,S)$ of Cayley graphs is a quasi-isometry.  The ``only if'' direction is due to Hopf and will be proved here in Theorem \ref{T:twoends}.
\end{rmk}

\noindent While the proof of Theorem \ref{T:numberofends} given above has the advantage of being rather elementary, we can give an even shorter proof once we establish the following general lemma about the nonstandard extension of the end stabilizer of a group.

\begin{lemma}\label{L:Hstar}
Let $W\subseteq G_{\inf}$ be internal.  Then there is $\nu\in \n^*\setminus \n$ such that $hx\in G_{\inf}$ and $hx\propto x$ for all $x\in W$ and all $h\in H^*$ with $|h|\leq \nu$.
\end{lemma}

\begin{proof}
Let $A_n:=\{\eta \in \n^* \ | \ \eta>n\}$, and for $h\in H$, let $$B_h:=\{\eta \in \n^* \ | \ hx\propto_\eta x \text{ for all }x\in W\}.$$  For each $h\in H$, we have $\n \subseteq B_h$, so the family $$\{A_n \ | \ n\in \n\} \cup \{B_h \ | \ h\in H\}$$ is a family of internal sets with the finite intersection property, so by saturation, there is $\gamma \in \bigcap_n A_n \cap \bigcap_h B_h$.  Consider the internal set 
$$C:=\{\eta \in \n^* \ | \ (\forall h\in H^*)(\forall x\in W) \  (|h|\leq \eta \to hx\propto_\gamma x)\}.$$  Since $\n\subseteq C$, there is $\nu \in \n^*\setminus \n$ with $\nu\in C$.  This $\nu$ is as desired.      
\end{proof}

Here now is a shorter proof of Theorem \ref{T:numberofends}.  Let $x_1,x_2\in G_{\inf}$ be such that $x_1\not \propto x_2$.  Fix $\nu \in\n^* \setminus \n$ such that $hx_i\propto x_i$ for $i=1,2$ and all $h\in H^*$ with $|h|\leq \nu$.  Fix $h\in H_{\inf}$ such that $|h|\leq \nu$ and such that $h\not \propto x_1$ and $h\not \propto x_2$; this is possible since there is $K\in \n$ such that every element of $G$ is within a distance of $K$ from an element of $H$, whence every element of $G^*$ is within a distance of $K$ from an element of $H^*$.  The proof now proceeds as in the final paragraph of the proof given above.  (The fact that $|x_1|=|x_2|$ was irrelevant in the proof of Theorem \ref{T:numberofends} given above.  Of course, one could take $x_1,x_2\in G_{\inf}$ such that $x_1\not\propto x_2$ and $|x_1|=|x_2|$.)

Let us mention one more application of Lemma \ref{L:asympcay}.

\begin{lemma}\label{L:prodinf}
If $G_1$ and $G_2$ are infinite, finitely generated groups, then $G_1\times G_2$ has one end.
\end{lemma}

\begin{proof}
We must show that if $(g_1,g_2)$ and $(h_1,h_2)$ are in $(G_1\times G_2)_{\inf}$, then $(g_1,g_2)\propto (h_1,h_2)$.  First suppose that $g_1=h_1\in (G_1)_{\inf}$.  Write $g_2=h_2s_1\cdots s_\nu$, where $s_i\in S^{\pm 1}$ for each $i\leq \nu$; here $S$ denotes the generating set for $G_2$.  Then $(g_1,g_2)=(h_1,h_2)\cdot (1,s_1)\cdots (1,s_\nu)$ and each initial segment $(h_1,h_2)\cdot (1,s_1)\cdots (1,s_i)$ is certainly in $(G_1\times G_2)_{\inf}$, whence $(g_1,g_2)\propto (h_1,h_2)$ by Lemma \ref{L:asympcay}.  An identical proof treats the case that $g_2=h_2\in (G_2)_{\inf}$.  Now suppose that $g_1\in (G_1)_{\inf}$ and $h_2\in (G_2)_{\inf}$.  Then by the special cases just treated above, we have that $$(g_1,g_2)\propto (g_1,h_2) \propto (h_1,h_2).$$  Finally, suppose that $g_1,h_1\in (G_1)_{\inf}$.  Fix $x\in (G_2)_{\inf}$.  Then $$(g_1,g_2)\propto (g_1,x)\propto (h_1,x)\propto (h_1,h_2).$$
\end{proof}

\begin{rmk}
The preceding lemma actually appears in \cite{C} as a corollary of the following more general result:  If $G$ contains an infinite, finitely generated normal subgroup $H$ such that $G/H$ is infinite, then $G$ has one end.  The proof of this fact is a rather straightforward combinatorial argument, and we were unable to find a nonstandard one simpler than it.  
\end{rmk}

\noindent We end this section with a short discussion of amalgamated free products and HNN extensions.  This material will be needed in the next section.

\begin{df}

\

\begin{enumerate}
\item Suppose that $G_1$ and $G_2$ are groups with subgroups $H_1$ and $H_2$ respectively.  Further suppose that $\phi:H_1\to H_2$ is an isomorphism.  Then \textbf{the amalgamated free product of $G_1$ and $G_2$ with respect to $\phi$} is the group
$$G_1*_{\phi}G_2:=\langle G_1,G_2 \ | \ \phi(h)h^{-1}, \ h\in H_1\rangle.$$  A more common notation for this amalgamated free product is $G_1*_H G_2$, where $H$ is a group isomorphic to both $H_1$ and $H_2$.
\item Suppose that $G$ is a group, $H_1$ and $H_2$ are subgroups of $G$, and $\phi:H_1\to H_2$ is an isomorphism.  Then the \textbf{HNN extension of $G$ via $\phi$} is the group
$$G*_{\phi}:=\langle G,t \ | \ tht^{-1}\phi(h)^{-1}, h\in H_1\rangle,$$ where $t$ is an element not in $G$, called the \emph{stable letter} of $G*_{\phi}$.  A more common notation for the HNN extension of $G$ via $\phi$ is $G*_H$, where $H$ is a group isomorphic to both $H_1$ and $H_2$.  
\end{enumerate}
\end{df}

The following theorem is considered one of the most important theorems in the theory of ends of finitely generated groups.  It would be a triumph to find a simple, nonstandard proof of this theorem.

\begin{thm}[Stallings \cite{Stall}, Bergman\cite{Berg}]
$G$ has more than one end if and only if one of the following holds:
\begin{itemize}
\item $G\cong A*_C B$, where $C$ is a finite group and $A\not=C$ and $B\not=C$, or
\item $G\cong A*_C$, where $C$ is a finite subgroup of $A$.
\end{itemize}
\end{thm}

We end with the Reduced Form Theorems for amalgamated free products (see \cite{Mag}, Theorem 4.1) and HNN extensions (see \cite{Brit} and \cite{Cohen}, Theorem 32).  The reduced form theorem for HNN extensions is also referred to as \emph{Britton's Lemma}.

\begin{fact}

\

\begin{enumerate}
\item Suppose that $C$ is a common subgroup of the groups $A$ and $B$.  Then every element $g\in A*_C B$ can be written in a \emph{reduced form}
$$g=cg_1\cdots g_n,$$ where $c\in C$, $g_1,\ldots,g_n\in (A\cup B)\setminus C$, and for all $i\in \{1,\ldots, n-1\}$, $g_ig_{i+1}\notin A\cup B$.  Moreover, the number $n$ is uniquely determined by $g$ and is called the \emph{length of $g$}, denoted $\ell(g)$.
\item Suppose that $A$ is a group and $\phi:C_1\to C_2$ is an isomorphism between two subgroups of $A$.  Let $t$ be the stable letter of $A*_\phi$.  Then every element $g\in A*_\phi$ can be written in a \emph{reduced form} $$g=g_0t^{\epsilon_1}g_1\cdots t^{\epsilon_n}g_n,$$ where $\epsilon_i\in \{-1,1\}$ for all $i\in \{1,\ldots,n\}$, $g_i\in A$ for all $i\in \{0,\ldots,n\}$, and there are no subwords of the form $t^{-1}a_it$ with $a_i\in C_1$ or $ta_it^{-1}$ with $a_i\in C_2$.  Moreover, the number $n$ is uniquely determined by $g$ and is called the \emph{length of $g$}, denoted $\ell(g)$.
\end{enumerate}
\end{fact}

A nonstandard consequence of this fact is that if $G$ is an amalgamated free product or HNN extension and $g\in G^*$ is such that $\ell(g)\in \n^*\setminus \n$, then $g\in G_{\inf}$.

\section{Groups with at Least Two Ends}\label{two}

In this section, we continue to let $X$ denote the metric space associated to $\operatorname{Cay}(G,S)$.  We further suppose that $G$ has at least two ends.  We fix $N$ such that $X\setminus B(1,N)$ has at least two unbounded path components.  $V$ will always denote the set of vertices of an unbounded path component of $X\setminus B(1,N)$.  Following Cohen \cite{C}, call $E\subseteq G$ \emph{almost invariant} if the symmetric difference $Eg\triangle E$ is finite for all $g\in G$.

\begin{lemma}\label{L:invariant1}
$V$ is almost invariant.
\end{lemma}

\begin{proof}
Fix $g\in G$ and $h\in V_{\inf}$.  Note that $hg^{-1}, hg \propto h$, whence $hg^{-1},hg\in V^*$.  It follows by underflow that for all $h\in V$ with $|h|$  sufficiently large, one has that $h\in Vg$ and $hg\in V$.       
\end{proof}

\begin{fact}\label{L:invariant2}
For all but finitely many $g\in V$, we have $gV\subseteq V$ or $G\setminus V \subseteq gV$.  
\end{fact}

\begin{proof}
This is actually a special case of Lemma 1.4 from \cite{C}, which states that given any two almost invariant subsets $E_1$ and $E_2$ of $G$, then for all but finitely many $g\in E_1$, one has either $gE_2\subseteq E_1$ or $G\setminus E_1\subseteq gE_2$.  (Cohen's Lemma 1.4 has a rather straightforward proof and we have been unable to find a nonstandard proof simpler than his.)  Taking $E_1=E_2=V$, which is almost invariant by Lemma \ref{L:invariant1}, we see that for almost all $g\in V$, either $gV\subseteq V$ or $G\setminus V\subseteq gV$.
\end{proof}

\noindent Recall that $H$ denotes the end stabilizer of $G$.

\begin{lemma}\label{L:invariant3}
For any $g\in H$, we have $gV\triangle V$ is finite.
\end{lemma}

\begin{proof}
Fix $g\in H$ and $h\in V_{\inf}$.  Then since $gh,g^{-1}h\propto h$, we have $gh,g^{-1}h\in V^*$.  So by underflow, we have that for $h\in V$ with $|h|$ sufficiently large, we have $gh,g^{-1}h\in V$, finishing the proof. 
\end{proof}

\begin{cor}\label{L:invariant5}
For all but finitely many $g\in V\cap H$, we have $gV\subseteq V$.
\end{cor}

\begin{proof}
By Fact \ref{L:invariant2}, for all but finitely many $g\in V\cap H$, we have $gV\subseteq V$ or $G\setminus V\subseteq gV$.  However, since $G\setminus V$ is infinite (as $G$ has at least two ends), the latter alternative contradicts Lemma \ref{L:invariant3}.
\end{proof}

\begin{cor}\label{L:invariant4}
For any $g\in V_{\inf}\cap H^*$, one has $gV^*\subseteq V^*$.
\end{cor}

\noindent The proof of the following theorem is essentially the same as in \cite{C}, but we include it here for completeness.

\begin{thm}[Hopf \cite{Hopf}, Cohen \cite{C}]\label{T:twoends}
If $G$ has at least two ends and has infinite end stabilizer, then $G$ is virtually $\z$ (whence it has exactly two ends).  In particular, if $G$ has exactly two ends, then $G$ is virtually $\z$.
\end{thm}

\begin{proof}
Fix $V$ as in the beginning of this section.  Choose $g\in V\cap H$ such that $gV\subseteq V$; this is possible by Corollary \ref{L:invariant5} and the fact that $H$ is infinite.  Note that then $g^n\in V$ for all $n$ (whence $g$ has infinite order) and that $g^{-1}\notin V$.  Now note that every $x\in V$ can be written as $x=g^mv$, for some $m$ and some $v\in V\setminus gV$.  Indeed, if $x\in \bigcap_n g^nV$, then $g^{-n}\in Vx^{-1}$ for all $n$; but Lemma \ref{L:invariant1} tells us that $Vx^{-1}$ differs from $V$ by a finite number of elements of $G$, yielding a contradiction to the fact that $g^{-n}\notin V$ for all $n$.  Likewise, since $G\setminus V$ is almost invariant, every $x\in G\setminus V$ can be written in the form $x=g^{-m}v$, for some $m$ and some $v\in (G\setminus V)\setminus (g^{-1}(G\setminus V))$.  Lemma \ref{L:invariant3} tells us that  $V\setminus gV$ is finite (whence $(G\setminus V)\setminus (g^{-1}(G\setminus V))$ is also finite), and hence the subgroup of $G$ generated by $g$ has finite index in $G$.
\end{proof}

\noindent Note that in a group with one end, we have $G=H$ and $g\propto g^{-1}$ for every $g\in G_{\inf}=H_{\inf}$.  Contrast this with the following lemma.

\begin{lemma}\label{L:invasymp}
Suppose that $G$ has two ends.  Then for all $g\in H_{\inf}$, $g\not\propto g^{-1}$.
\end{lemma}

\begin{proof}
Consider $g\in H_{\inf}$ and fix $V$ such that $g\in V^*$.  By Corollary \ref{L:invariant4}, we have $gV^*\subseteq V^*$.  If $g\propto g^{-1}$, then $g^{-1}\in V^*$, whence $1\in V^*$, a contradiction.
\end{proof}

\begin{lemma}\label{L:semigroup}
Suppose that $G$ has two ends.  Then for every hyperfinite sequence $g_1,\ldots,g_\eta$ of elements of $H_{\inf}$ such that $g_i\propto g_j$ for all $i,j\in \{1,\ldots,\eta\}$, we have $g_1\cdots g_\eta\in H_{\inf}$ and $g_1\cdots g_\eta\propto g_1$.  In particular, for every $g\in H_{\inf}$ and every $\eta\in \n^*\setminus \{0\}$, we have $g^\eta\in H_{\inf}$ and $g^\eta\propto g$.
\end{lemma}

\begin{proof}
Let $V$ be such that $g_i\in V^*$ for all $i\in \{1,\ldots,\eta\}$.  By Corollary \ref{L:invariant4}, we have $g_iV^*\subseteq V^*$ for each $i\in \{1,\ldots,\eta\}$.  By internal induction, one can show that $g_{\eta-i}\cdots g_\eta\in V^*$ for all $i\in \{0,\ldots,\eta-1\}$, whence $g_1\cdots g_\eta \in V^*$.  Hence $g_1\cdots g_\eta \propto_N g_1$.  Notice that the same argument can be applied to any $n\geq N$, whence $g_1\cdots g_\eta \propto_n g_1$ for all $n\geq N$.  Hence, by overflow, there is $\nu \in \n^*\setminus \n$ such that $g_1\cdots g_\eta \propto_\nu g_1$, finishing the proof.  
\end{proof}

\begin{ex}\label{L:free}
The free product $G:=\z_2*\z_2$ is a group with two ends which does not equal its own end stabilizer; here $\z_2$ denotes the group of two elements.  Let $a$ and $b$ be distinct generators for the two factors of $\z_2$.  To see that $G$ has two ends, notice that reduced words of infinite length are in the same infinite path component if and only if they both begin with $a$ or both begin with $b$.  It then follows that left multiplication by $a$ permutes the two ends of $G$, so $G$ does not equal its own end stabilizer.  Another way to see that $G$ does not equal its own end stabilizer is the observation that any reduced word of infinite length which begins and ends with the same element (e.g. $\underbrace{abab\cdots a}_{\nu \text{ factors, }\nu \in \n^*\setminus \n}$) has order $2$, whence cannot be in the nonstandard extension of the end stabilizer by Lemma \ref{L:semigroup}.
\end{ex}

\noindent Lemma \ref{L:semigroup} leads us to ask what groups $G$ have \textbf{multiplicative ends}: for all infinite $g,g'\in G^*$, if $g\propto g'$, then $gg'\in G_{\inf}$ and $gg'\propto g$?  It turns out that there is a standard characterization of groups with this property.  We first provide a well-known consequence of Stalling's Theorem for which we were unable to find a reference.  The outline of the proof was communicated to me by Ilya Kapovich.  Recall that a group $G$ is a \emph{(internal) semidirect product of $K$ by $Q$} if $K$ and $Q$ are subgroups of $G$, $K$ is normal in $G$, $G=KQ$, and $K\cap Q=\{1\}$.

\begin{lemma}\label{L:chartwoends}
A finitely generated group $G$ has two ends if and only if $G$ is a semidirect product of a finite group by a group which is isomorphic to $\z$ or $\z_2*\z_2$.  
\end{lemma}

\begin{proof}
The ``if" direction follows from Example \ref{L:free} and the fact that a virtually two-ended group is itself two-ended.  We now prove the ``only if" direction.  By Stallings theorem, $G$ admits a simplicial cocompact action on a simplicial line $T$ with finite-edge stabilizers.  We thus obtain a homorphism $\alpha:G\to \operatorname{Aut}(T)$ with finite kernel $K$.  We claim that $\alpha(G)$ is isomorphic to $\z$ or $\z_2*\z_2$.

\noindent \textbf{Case 1:}  $\alpha(G)$ only contains translations.  Choose $g\in G$ such that the translation distance of $\alpha(g)$ is minimal with respect to the translation distances of the elements of $\alpha(G\setminus K)$.  We claim that $\alpha(G)=\langle \alpha(g)\rangle$, the subgroup of $\Aut(T)$ generated by $\alpha(g)$, yielding that $\alpha(G)$ is isomorphic to $\z$.  Indeed, let $n>0$ equal the translation distance of $\alpha(g)$.  Fix $h\in G\setminus K$ and let $m$ equal the translation distance of $\alpha(h)$.  Let $q,r\in \n$ be such that $m=qn+r$, where $q>0$ and $r\in \{0,\ldots,n-1\}$.  Since $\alpha(g^{-q}h)$ is an element of $\alpha(G)$ of translation distance $r$, it follows by choice of $g$ that $\alpha(g^{-q}h)=\operatorname{id}_T$ and hence $\alpha(h)$ is in the subgroup of $\Aut(T)$ generated by $\alpha(g)$.

\noindent \textbf{Case 2:}  $\alpha(G)$ contains an orientation-reversing element $\alpha(g)$.  We first claim that $\alpha(G)$ also contains a nontrivial translation.  Since $\alpha$ has a finite kernel, we have that $\alpha(G)$ is infinite.  Choose $h\in G$ such that $\alpha(h)\notin \{\operatorname{id}_T,\alpha(g)\}$.  If $\alpha(h)$ is not a translation, then $\alpha(h)$ is an orientation-reversing element, whence $\alpha(g)\alpha(h)=\alpha(gh)$ is a nontrivial translation.  Choose $h\in G$ such that $\alpha(h)$ is a nontrivial translation and the translation distance of $\alpha(h)$ is minimal with respect to the translation distances of the translations in $\alpha(G\setminus K)$.  Let $g':=gh$.  We next claim that $\alpha(G)=\langle \alpha(g),\alpha(g')\rangle$, the subgroup of $\Aut(T)$ generated by $\alpha(g)$ and $\alpha(g')$.  Fix $y\in G\setminus K$.  If $\alpha(y)$ is a translation, then $\alpha(y)\in \langle \alpha(h)\rangle$ as in Case 1.  If $\alpha(y)$ is an orientation reversing element, then $\alpha(gy)$ is a translation, whence $\alpha(y)\in \langle \alpha(g),\alpha(h)\rangle$.  Now it is easy to prove that the natural map $\langle g\rangle*\langle g'\rangle \to \alpha(G)$ is an isomorphism, whence it follows that $\alpha(G)$ is isomorphic to $\z_2* \z_2$.   

In either case, the exact sequence $1\to K \to G \to \alpha(G) \to 1$ admits a splitting, i.e. a group homomorphism $\beta:\alpha(G)\to G$ such that $\alpha \beta=\operatorname{id}_{\alpha(G)}$.  It then follows that $G$ is a semidirect product of $K$ by $\alpha(G)$.
\end{proof}

\begin{prop}\label{L:multends}
For a finitely generated group $G$, the following are equivalent:

\begin{enumerate}
\item $G$ has two ends and equals its own end stabilizer;
\item $G$ has multiplicative ends;
\item for all $g\in G_{\inf}$, $g\not \propto g^{-1}$;
\item $G$ is a semidirect product of a finite group by an infinite cyclic group;
\item $G$ has two ends and has an infinite cyclic central subgroup.
\end{enumerate}
\end{prop}

\begin{proof}
$(1)\Rightarrow (2)$  is immediate from Lemma \ref{L:semigroup}.  

$(2)\Rightarrow (3)$ is trivial.

$(3)\Rightarrow (1)$:  $G$ cannot have one end, for then we have $g\propto g^{-1}$ for all $g\in G_{\inf}$.  Now suppose that $G$ has two ends, but is not equal to its own end stabilizer.  Then $H$ has index $2$ in $G$, say $G=H\sqcup xH$.  Let $h$ be in $H_{\inf}$ and set $g:=xh$.  Then, by Lemma \ref{L:invasymp}, we have $h\not\propto h^{-1}$, and since $x$ permutes the two ends of $G$, it follows that $g=xh\propto h^{-1}\propto h^{-1}x^{-1}=g^{-1}$, whence (3) fails.  It remains to eliminate the case that $G$ has infinitely many ends.  By Stalling's Theorem, we know that $G$ is either isomorphic to an amalgamated free product $A*_C B$ or an HNN extension $A*_C$, where $C$ is finite, $[A:C]\geq 3$, and $[B:C]\geq 2$.  We show that both of these situations contradict (3).  First consider the case of the amalgamated free product $G=A*_C B$.  Without loss of generality, we suppose that $C$ is a common subgroup of $A$ and $B$.  Fix $a\in A\setminus C$ and $b\in B\setminus C$.  Fix $\nu \in \n^*\setminus \n$.  Let $g:=\underbrace{abab\cdots a}_{\nu \text{ factors}}$.  By the reduced form theorem for amalgamated free products, we have that that the elements $gbg^{-1}$ and $gb^{-1}g^{-1}$ are both infinite and in the same infinite path component as $g$.  Hence, $gbg^{-1}\propto (gbg^{-1})^{-1}$, contradicting (3).  Now consider the case of the HNN extension $G=A*_\phi$, where $\phi:C_1\to C_2$ is an isomorphism between two subgroups of $A$.  Let $t$ be the stable letter of $G$.  Fix $a\in A\setminus C_1$ and $\nu \in \n^*\setminus \n$.   Then by Britton's Lemma, the elements $t^{-\nu} a t^\nu$ and $t^{-\nu} a^{-1} t^\nu$ are both infinite and in the same infinite path component as $t^{-\nu}$, yielding $t^{-\nu} a t^{\nu}\propto (t^{-\nu} a t^{\nu})^{-1}$, contradicting (3).

$(1)\Rightarrow (4)$ If $G$ were isomorphic to a semidirect product of a finite group by $\z_2*\z_2$, then $G$ has infinite elements of order $2$ (see the argument of Example \ref{L:free}), whence $G$ does not equal its own end stabilizer by Lemma \ref{L:semigroup}.  By Lemma \ref{L:chartwoends}, $G$ must be isomorphic to a semidirect product of a finite group by $\z$.

$(4)\Rightarrow (5)$:  Suppose $G$ is a semidirect product of the finite group $K$ by an infinite cyclic group $L$ with generator $l$.  Since $G$ is virtually $\z$, we know that $G$ has two ends.  Since conjugation by $l$ is an automorphism of $K$ and $K$ is finite, there must be $n$ such that $l^nkl^{-n}=k$ for all $k\in K$.  It follows that $l^n$ is central in $G$ (and has infinite order).

$(5)\Rightarrow (1)$:  Suppose that $G$ has two ends and has an infinite cyclic central subgroup $L$ generated by $l$.  Suppose, towards a contradiction, that $G$ is not equal to its own end stabilizer.  Choose $x\in G$ such that $x$ permutes the two ends of $G$.  Fix $\nu \in \n^*\setminus \n$.  Then $xl^\nu=l^\nu x \propto l^\nu$, contradicting the fact that $x$ permutes the ends of $G$.       
\end{proof}

As is well-known, semidirect products are sensitive to the order of the factors; the next lemma exemplifies this fact.

\begin{lemma}
If $G$ is a semidirect product of an infinite cyclic group $L$ by a finite group $K$, then $G$ is a finitely generated group with two ends which is \textbf{not} equal to its own end stabilizer unless $G$ is  the direct product of $L$ and $K$.
\end{lemma} 

\begin{proof}
If $G$ is not isomorphic to the direct product of $L$ and $K$, there must be $k\in K$ such that for every $l\in L$, $klk^{-1}=l^{-1}$.  Let $l\in \z_{\inf}$ be arbitrary.  Then $$(lk)\cdot (lk)=l(klk^{-1})k^2=k^2\in G,$$ whence $G$ does not have multiplicative ends, and hence, by Lemma \ref{L:multends}, $G$ is not equal to its own end stabilizer.
\end{proof}

Now we consider the situation when $G$ has infinitely many ends, whence the end stabilizer $H$ of $G$ is \emph{finite}.  Let $G/H$ denote the set of right cosets of $H$ in $G$.  In general, we have $(G/H)^*=G^*/H^*$.  Since $H$ is finite, we have $H^*=H$, so $(G/H)^*=G^*/H$.  Let us assume that $S=S^{-1}$ and let $\tilde{S}:=S\setminus H$.  Note that the image of $\tilde{S}$ under the natural map $G\to G/H$ is a generating set for $G/H$ not containing the trivial coset $H$.  Let $\tilde{X}:=\operatorname{Cay}(G/H,\tilde{S})$.  As before, we have that $(G/H)_{\f}=G/H$ and hence that $(G/H)_{\inf}=G_{\inf}/H$.

\begin{lemma}
If $G$ has infinitely many ends, then $G/H$ has trivial end stabilizer.
\end{lemma}  

\begin{proof}
Suppose $g\in G$ is such that $Hg$ fixes the ends of $\operatorname{IPC}(\tilde{X})$.  Fix $g'\in G_{\inf}$.  Then by hypothesis, we have $Hg'\propto Hgg'$, so there are $s_0,\ldots,s_\nu\in \tilde{S}$ such that $Hgg'=Hg's_0\cdots s_\nu$ and satisfying $Hg's_0\cdots s_i\in (G/H)_{\inf}$ for every $i\in \{0,\ldots,\nu\}$.  Write $gg'=hg's_0\cdots s_\nu$, where $h\in H$.  It now follows that  
$$gg'=hg's_0\cdots s_\nu \propto g's_0\cdots s_\nu \propto g'.$$  Since $g'\in G_{\inf}$ was arbitrary, we have that $g\in H$, completing the proof of the lemma.    
\end{proof}

\begin{cor}
If $G$ has infinitely many ends, then $\ee(G)$ is homeomorphic to $\ee(G')$, where $G'$ is a group with trivial end stabilizer.
\end{cor}

\begin{proof}
Since $H$ is a finite normal subgroup of $G$, $\operatorname{Cay}(G)$ and $\operatorname{Cay}(G/H)$ are quasi-isometric, whence $\ee(G)$ and $\ee(G/H)$ are homeomorphic.  Take $G'=G/H$.
\end{proof}

To summarize, if $G$ is a group with at least two ends, then $G$ has infinite end stabilizer if and only if $G$ has exactly two ends.  If $G$ has infinitely many ends, then we know that its end stabilizer must be finite, and then in this case, $G$ is quasi-isometric with a group with trivial end stabilizer.

\section{Relative Ends}

In this section, we still assume that $G$ is an infinite, finitely generated group with finite generating set $S$.  We further suppose that $K$ is a subgroup of $G$ of \emph{infinite index in }$G$.  We let $\operatorname{Cay}(G,K,S)$, the \textbf{relative Cayley graph of $G$ with respect to $K$ and $S$}, be the locally finite graph with $\V=G/K$, the set of \emph{right cosets} of $K$ in $G$, and such that $(Kg,Kg')\in \E$ if there is $s\in S^{\pm 1}$ such that $Kg'=Kgs$.  We let $X$ denote the metric space obtained from $\operatorname{Cay}(G,K,S)$.  As in the case of the ordinary Cayley graph, if $S'$ is also a finite generating set for $G$ and $X'$ is the metric space obtained from $\operatorname{Cay}(G,K,S')$, then $X$ and $X'$ are quasi-isometric, whence we can speak of $\ee(G,K)$ and $\ipc(G,K)$ as the spaces of ends and infinite path components of any relative Cayley graph of $G$ with respect to $K$.

 Since $G\cap K^*=K$, the natural map $$\iota:G/K\to (G/K)^*=G^*/K^*, \quad \iota(Kg)=K^*g,$$ is injective.  Note that for $g\in G^*$, $K^*g\in (G/K)_{\f}$ if and only if there are $s_1,\ldots,s_n\in S^{\pm 1}$ such that $K^*g=K^*s_1\cdots s_n$, that is $(G/K)_{\f}=\iota(G/K)$.  In other words, $K^*g\in (G/K)_{\f}$ if and only if there is $x\in G$ such that $gx\in K^*$.  This leads to the following definitions.

\begin{df}

\

\begin{enumerate}
\item $G_{\f, K}:=\{g\in G^* \ | \text{ there exists } x\in G \text{ such that }gx\in K^*\}$;
\item $G_{\inf,K}=G^*\setminus G_{\f,K}=\{g\in G^* \ | \text{ for all }x\in G \text{ we have }gx\notin K^*\}$.
\end{enumerate}
\end{df}

\noindent Note that $G\subseteq G_{\f,K}$ and $G_{\f,K}=G$ (whence $G_{\inf,K}=G_{\inf}$) if and only if $K$ is finite.  These definitions were made so that the identities $$(G/K)_{\f}=\{K^*g \ | \ g\in G_{\f,K}\}$$ and $$(G/K)_{\inf}=\{K^*g\ | \ g\in G_{\inf,K}\}$$ would hold tautologically.

\begin{lemma}

\

\begin{enumerate}
\item $K^*\cdot G_{\f,K}\subseteq G_{\f,K}$;
\item $K^*\cdot G_{\inf,K}\subseteq G_{\inf,K}$.
\end{enumerate}
\end{lemma}     

\begin{proof}
For (1), suppose that $g\in G_{\f,K}$, so there is $x\in G$ such that $gx\in K^*$.  But then if $h\in K^*$, we have $(hg)x=h(gx)\in K^*$, whence $hg\in G_{\f,K}$.  (2) follows easily from (1).  
\end{proof}

\noindent The following is immediate from Lemma \ref{L:asympgraph}.

\begin{lemma}
Suppose that $K^*g, K^*g'\in (G/K)_{\inf}$ (so $g,g'\in G_{\inf,K}$).  Then $K^*g\propto K^*g'$ if and only if there is a hyperfinite sequence $s_0,\cdots,s_\nu\in S^{\pm 1}$ such that $K^*gs_0\cdots s_\nu=K^*g'$ and $K^*gs_0\cdots s_i\in (G/K)_{\inf}$ for all $i\in \{1,\ldots,\nu\}$.   
\end{lemma}

We now formulate the relation $\propto$ for $(G/K)^*$ in terms of a related notion in $G^*$.

\begin{df}
For $g,g'\in G_{\inf,K}$, define $g\propto_K g'$ if there exists a hyperfinite sequence $s_0,\cdots,s_\nu \in S^{\pm 1}$ such that $gs_0\cdots s_\nu=g'$ and $gs_0\cdots s_i\in G_{\inf,K}$ for all $i\in\{1,\ldots,\nu\}$.
\end{df}

Note that the relation $\propto_K$ is an equivalence relation, whence we can speak of the \textbf{$K$-infinite path components of $G^*$}.  Note that $g\propto_Kg'$ implies that $g\propto g'$, and if $K$ is finite, then the notion $\propto_K$ is just the notion $\propto$.  The definitions were made so that the following lemma would be a tautology.

\begin{lemma}
For $g,g'\in G_{\inf,K}$, we have $K^*g\propto K^*g'$ if and only if there exists $h\in K^*$ such that $g\propto_K hg'$.
\end{lemma}

Let $N_G(K):=\{g\in G \ | \ ghg^{-1}\in K \text{ for all }h\in K\}$ be the normalizer of $K$ in $G$.  Notice that $N_G(K)/K$ acts on $G/K$ by left multiplication and this action preserves the relation $\E$, whence we can extend this action to an isometry of $X$.  By Lemma \ref{L:qicont}, we obtain a group homomorphism $$Kg\mapsto ([K^*x]\mapsto [K^*gx]):N_G(K)/K\to \operatorname{Homeo}(\ipc(X)).$$ Let $L$ be the normal subgroup of $N_G(K)$ such that $L/K$ is the kernel of the above morphism, so for $l\in L$ and $K^*x\in (G/K)_{\inf}$, we have $K^*lx\propto K^*x$.

For $x\in G$, let $|Kx|:=d(K,Kx)$ so $K^*x\in (G/K)_{\f}$ if and only if $|K^*x|\in \n$.

\begin{lemma}
Let $W\subseteq (G/K)_{\inf}$ be internal.  Then there is $\nu\in \n^*\setminus \n$ such that $K^*lx\in (G/K)_{\inf}$ and $K^*lx\propto K^*x$ for all $K^*x\in W$ and all $K^*l\in (L/K)^*$ with $|K^*l|\leq \nu$.
\end{lemma}

\begin{proof}
Exactly like the proof of Lemma \ref{L:Hstar}.
\end{proof}

\begin{thm}[see \cite{Geog}, Theorem 13.5.21]
If $N_G(K)/K$ is infinite and $\ee(G,K)$ is finite, then $|\ee(G,K)|\leq 2$.
\end{thm}

\begin{proof}
Suppose, towards a contradiction, that $3\leq |\ee(G,K)|<\infty$.  Choose $K^*x_1, K^*x_2\in (G/K)_{\inf}$ such that $K^*x_1\not\propto K^*x_2$.  Since $L/K$ has finite index in $N_G(K)/K$, we must have that $L/K$ is infinite.  Choose $\nu\in \n^*\setminus \n$ such that $K^*lx_i\propto K^*x_i$ for $i=1,2$ and all $K^*l\in (L/K)^*$ with $|K^*l|\leq \nu$.  Choose $K^*l\in (L/K)^*\cap (G/K)_{\inf}$ such that $|K^*l|\leq \nu$ and such that $K^*l\not\propto K^*x_1$ and $K^*l\not \propto K^*x_2$.  Write $K^*x_1=K^*s_0\cdots s_\eta$, $K^*x_2=K^*t_0\cdots t_\zeta$, where $\eta,\zeta\in \n^*\setminus \n$, each $s_i,t_j\in S^{\pm 1}$, and such that $|K^*s_0\cdots s_i|=i+1$ and $|K^*t_0\cdots t_j|=j+1$ for all $i\in \{1,\ldots,\eta\}$ and all $j\in \{1,\ldots,\zeta\}$.  Since $K^*lx_1\propto K^*x_1\not\propto K^*l$, we must have $K^*ls_0\cdots s_i\in (G/K)_{\f}$ for some $i<\eta$.  Similarly, $K^*lt_0\cdots t_j\in (G/K)_{\f}$ for some $j<\zeta$.  We now must have $g\in G$ such that $ls_0\cdots s_ig^{-1}t_j^{-1}\cdots t_0^{-1}l^{-1}\in K^*$.  Since $l\in N_G(K)^*$, we have $K^*s_0\cdots s_i=K^*t_0\cdots t_jg$.  Since $l\in (L/K)_{\inf}$, we must have $i,j\in\n^*\setminus\n$, whence we have $$K^*x_1\propto K^*s_0\cdots s_i\propto K^*t_0\cdots t_jg\propto K^*t_0\cdots t_j \propto K^*x_2,$$ which is a contradiction.
\end{proof}

\noindent Much less is known about spaces of relative ends than about ordinary end spaces.  Perhaps nonstandard reasoning will be useful in further studying relative end spaces.


\begin{thebibliography}{1}

\bibitem{Berg} G. Bergman, \textit{On groups acting on locally finite graphs}, Annals of Math. 88 (1968), pp. 335-340.
\bibitem{BH} M. Bridson and A. Haefliger, \textit{Metric Spaces of Non-Positive Curvature}, Fundamental Principles of Mathematical Sciences 319, Springer-Verlag, Berlin, 1999.

\bibitem{Brit} J. L. Britton, \textit{The Word Problem}, Ann. of Math. 77 (1963), pp. 16-32.
\bibitem{C} D.E. Cohen, \textit{Ends and Free Products of Groups}, Math. Z. 114 (1970), pp. 9-18.
\bibitem{Cohen} D.E. Cohen \textit{Combinatorial Group Theory: A Topological Approach}, Cambridge Univ. Press, Cambridge, 1989.
\bibitem{D} M. Davis, \textit{Applied Nonstandard Analysis}, John Wiley and Sons Inc., 1977.
\bibitem{DW} L. van den Dries and A. Wilkie, \textit{Gromov's theorem on groups of polynomial growth and elementary logic}, J. Algebra 89 (1984), 349-374.

\bibitem{Geog} R. Geoghegan, \textit{Topological Methods in Group Theory}, Graduate Texts in Mathematics 243, Springer Science+Busines Media, LLC, New York, 2008.

\bibitem{Gold} I. Goldbring, \textit{Hilbert's fifth problem for local groups}, to appear in the Annals of Math.
\bibitem{He} C.W. Henson, \textit{Foundations of Nonstandard Analysis: A Gentle Introduction to Nonstandard Extensions}; Nonstandard Analysis: Theory and Applications, L. O. Arkeryd, N. J. Cutland, and C. W. Henson, eds., NATO Science Series C:, Springer, 2001.
\bibitem{Hopf} H. Hopf, \textit{Enden offener R\"aume und unendliche diskontinuierliche Gruppen}, Comment. Math. Helv., 16 (1943), pp. 81-100.

\bibitem{F} D. Farley, \textit{A Proof That Thompson's Groups Have Infinitely Many Ends}, preprint, arXiv 0708.1334.

\bibitem{H} H. Hopf, \textit{Enden ofener R\"aume und unendliche diskontinuerliche Gruppen}, Comment. Math. Helv., 16 (1943), pp. 81-100.
\bibitem{Mag} W. Magnus, A. Karras and D. Solitar, \textit{Combinatorial Group Theory}, Interscience Publishers, New York a.o., 1966.
\bibitem{Stall} J. Stallings, \textit{On torsion-free groups with infinitely many ends}, Ann. of Math., 88 (1968), pp 312-334.


\end{thebibliography}
\end{document}